\documentclass[12pt]{article}
\usepackage{graphicx}
\usepackage{tikz}
\tikzstyle{vertex}=[circle, draw, fill=black, inner sep=0pt, minimum size=4pt]
\usepackage{amsmath,amsthm}
\usepackage{latexsym}
\newtheorem{theorem}{Theorem}
\newtheorem{corollary}[theorem]{Corollary}
\newtheorem{proposition}[theorem]{Proposition}

\newcommand{\match}{\mathcal{M}}
\theoremstyle{definition}
\newtheorem{definition}[theorem]{Definition}
\mathchardef\mhyphen="2D

\title{General polygonal line tilings and their matching complexes}
\author{Margaret Bayer\\Department of Mathematics\\University of Kansas\\Lawrence, Kansas, U.S.A.\\bayer@ku.edu \and Marija Jeli\'{c} Milutinovi\'{c}\\Faculty of Mathematics\\University of Belgrade\\ Belgrade, Serbia\\marijaj@matf.bg.ac.rs \and  Julianne Vega\thanks{This article is based on work supported by the National Science Foundation under Grant No.\ DMS-1440140 while the authors participated in the 2020/2021 Summer Research in Mathematics Program of the Mathematical Sciences Research Institute, Berkeley, California. M.\ Jeli\'{c} Milutinovi\'{c} has been supported by the Project No.\ 7744592 MEGIC
”Integrability and Extremal Problems in Mechanics, Geometry and
Combinatorics” of the Science Fund of Serbia, and by the Faculty of Mathematics University of Belgrade through the grant (No.\ 451-03-68/2022-14/200104) by the Ministry of
Education, Science, and Technological Development of the Republic of Serbia.}
\\Department of Mathematics\\Maret School\\Washington, D.C. U.S.A.\\jvega@maret.org}

\begin{document}
\maketitle
\begin{abstract}
A (general) polygonal line tiling is a graph formed by a string of cycles, each intersecting the previous at an edge, no three intersecting. In 2022, Matsushita proved the matching complex of a certain type of polygonal line tiling with even cycles is homotopy equivalent to a wedge of spheres. In this paper, we extend Matsushita's work to include a larger family of graphs and carry out a closer analysis of lines of triangle and pentagons, where the Fibonacci numbers arise.
\end{abstract}

\section{Introduction}
For a finite simple graph $G$, the matching complex $\match(G)$ is the
simplicial complex on the set of edges with faces given by matchings in the graph, where a matching 
is a set of edges no two of which share a vertex.  The topology of
matching complexes has been the subject of much research over the years.
Chessboard complexes, which are the matching complexes of complete bipartite 
graphs, have been studied by many authors, including Athanasiadis~\cite{Athan}, Bj\"{o}rner, et al.\ \cite{blvz}, Joji\'{c} \cite{jojic},
Shareshian and Wachs \cite{shareshian}, and Ziegler \cite{ziegler-chessboard}.
See Wachs \cite{wachs} for a survey.  Other matching complexes that have
been studied include those for 
paths and cycles (Kozlov \cite{kozlov-directed}) and
trees (Marietti and Testa \cite{MT_forests} and Jeli\'{c} Milutinovi\'{c} et al.\ \cite{JelicEtAl}). 
Most relevant to this paper is the study of matching complexes of 
grid graphs (Braun and Hough \cite{Braun_Hough} and 
Matsushita \cite{matsushita2018}),
polygonal line tilings (Matsushita \cite{takahiro}), and
honeycomb graphs (Jeli\'{c} Milutinovi\'{c} et al.\ \cite{JelicEtAl}). 
We are particularly interested in graphs whose matching complexes are contractible or homotopy equivalent to a wedge of spheres.  These are not all graphs; for example, in \cite{blvz} it is shown that
the matching complex of the complete bipartite graph $K_{3,4}$ is a torus.

Other papers take different approaches to the study of matching complexes. Bayer, et al. \cite{bayer} start with the topology of the matching complex and identify the graphs that produce it. The current authors \cite{bjv-perfect} define the {\em perfect matching complex}, the subcomplex of the matching complex with facets corresponding to perfect matchings, and study this complex for honeycomb graphs.

Note that other simplicial complexes associated
with graphs have been studied from a topological viewpoint.  
See, in particular,
Jonsson's book \cite{jonsson-book}.  We will see that 
tools developed for
the study of independence complexes of graphs 
by Adamaszek \cite{adamaszek-split} and Engstr\"{o}m \cite{engstrom} play an important role in our work.

In  this paper we will focus on graphs that are formed from lines of polygons.
This expands on the work of Matsushita \cite{takahiro}, who studied
matching complexes of the graphs formed by lines of $2n$-gons, 
intersecting at parallel edges.  
Our main result is that any line of polygons, allowing different size
polygons in the line (as long as each has at least four edges), has matching
complex that is contractible or homotopy equivalent to a wedge of spheres.
We also consider lines of triangles, where we can specify the dimensions
and numbers of spheres in the wedge.  In the case of pentagonal line tilings
we give the explicit homotopy type (involving Fibonacci numbers).

\section{Overview}
 We introduce the definitions and propositions that we use throughout the remainder of the article. Let $G$ be a finite simple graph. 
\begin{definition}
A {\em matching} of a graph $G$ is a set of edges of $G$, no two of
which share a vertex. The {\em matching complex} of a graph $G$ is the simplicial complex
$\match(G)$ with vertex set $E$, the set of edges of $G$, and faces
the subsets $\sigma\subseteq E$ that form matchings of $G$.
\end{definition}

\begin{definition}
An {\em independent set} of a graph $G$ is a set of vertices of $G$,
no two of which form an edge. The {\em independence complex} of a graph $G$ is the simplicial complex
${\cal I}(G)$ with vertex set $V$, the set of vertices of $G$, and
faces the subsets $\sigma\subseteq V$ that form independent sets of $G$.
\end{definition}

\begin{definition}
The {\em line graph} $L(G)$ of a graph $G$ is the graph with vertex set
the set of edges of $G$ and edge set the set of pairs of edges of $G$
that share a vertex.
\end{definition}

The following statement follows directly from the definitions.
\begin{proposition} \label{prop:math_ind}
The matching complex of $G$ is the independence complex of $L(G)$.
\end{proposition}

 It is not true that every independence complex is a matching complex. 
 For example, consider the complex with facets $\{a,b,c\}$ and $\{d\}$.  This is the independence complex of $K_{1,3}$, but it is not the matching complex of any graph, since such a graph would have three independent edges and one edge that intersects all three of them.

Proposition~\ref{prop:math_ind} enables us to translate theorems about independence complexes to
theorems about matching complexes.  

Define the (open) edge neighborhood $EN_G(e)$ 
of an edge $e$ in the graph $G$ to be the 
set of edges adjacent to $e$, and 
the closed edge neighborhood of $e$ to be $EN_G[e] = EN_G(e) \cup \{e\}$. 
(When the graph $G$ is clear from context we write $EN(e)$ and $EN[e]$, respectively.)

For a simplex $\sigma$ in a simplicial complex $K$ the link of $\sigma$ is
\[ 
\textrm{lk}(\sigma, K) = \{ \tau \in K \mid \tau \cap \sigma = \emptyset, \tau \cup \sigma \in K\}  
\] 

and the (face) deletion of $\sigma \in K$ is 
\[
\textrm{del}(\sigma, K) = \{ \tau  \in K \mid \sigma \not\subset \tau\}. 
\]

For a graph $G$ and edge $e \in E(G)$, denote the corresponding vertex in $\match(G)$ as $\bar{e}$. Then 
$\textrm{lk}(\bar{e}, \match(G)) = \match(G \setminus EN_G[e])$. Since for a vertex $v$ of $K$ the sequence 
 ${\textrm{lk}(v, K) \rightarrow {\textrm{del}(v, K) \rightarrow K}}$ is a cofiber sequence (see \cite [Section 2]{adamaszek-split} for the definition of cofiber sequence and further details), we have the following result.

\begin{proposition}[Adamaszek \cite{adamaszek-split}, Proposition~3.1] \label{prop:cofiber}
The sequence
\[
\match(G \setminus EN[e]) \hookrightarrow \match(G \setminus \{e\}) \hookrightarrow \match(G)
\]
is a cofiber sequence. If the inclusion $\match(G \setminus EN[e]) \hookrightarrow \match(G \setminus \{e\})$ is null-homotopic, then there is a homotopy equivalence $\match(G) \simeq \match(G \setminus \{e\}) \vee \Sigma \match(G \setminus EN[e])$.
\end{proposition}

\begin{proposition}[Engstr\"om \cite{engstrom2}, Lemma 2.4]\label{prop:open_neighborhood}
Let $G$ be a graph that contains two different edges $e$ and $h$ such that $EN(e) \subset EN(h).$ Then $\match(G)$ collapses to  $\match(G \setminus\{h\})$. That is, $\match(G)\simeq \match(G \setminus\{h\})$.
\end{proposition}

\begin{proposition}[Adamaszek \cite{adamaszek-split}, Theorem~3.3]\label{prop:closed_neigh}
Let $G$ be a graph that contains two different edges $e$ and $h$ such that $EN[e] \subset EN[h].$ Then 
$$\match(G) \simeq \match(G  \setminus\{h \}) \vee \Sigma\match(G  \setminus EN[h]).$$
\end{proposition}

An edge $e$ in $G$ is simplicial if $L(G[EN(e)])$ is a complete graph. 
That is, every pair of edges adjacent to $e$ are themselves adjacent.

\begin{proposition}[Engstr\"om \cite{engstrom2}, Lemma~2.5]\label{prop:simplicial_edge}
If $e$ is a simplicial edge in $G$, then there is a homotopy equivalence 
\[
\match(G) \simeq \bigvee\limits_{w \in EN(e)} \Sigma \match(G \setminus EN[w]).
\]
\end{proposition}

\begin{corollary}\label{path_length_3} If $G$ is a graph, and if $P$ is a path of length 3 that intersects $G$ at just one endpoint, then $\match(G\cup P) \simeq \Sigma\match(G)$.
\end{corollary}

\begin{proposition}[Engstr\"om \cite{engstrom-KM}, Lemma~2.2]
\label{prop:shorten_path}
If $G$ is a graph with 
a path $X$ of length 4 whose internal vertices are of degree two and whose end vertices are distinct, 
then $\match(G) \simeq \Sigma \match(G/ X)$, where $G / X$ is the contraction of $X$ to a single edge with endpoints given by the endpoints of $X$.
\end{proposition}

The resulting contraction may have parallel edges. The following proposition explains the homotopy type of the matching complex in those situations. 
\begin{proposition}
\label{double_edge}
Let $G$ be a graph and $e$ an arbitrary edge in $G$. Consider a graph $G \cup \{x\}$ obtained by adding an edge $x$ parallel to $e$ ($x$ and $e$ have same endpoints). Then:
$$\match(G \cup \{x\}) \simeq \match(G) \vee \Sigma \match(G \setminus EN_{G}[e]).$$ 
\end{proposition}
\begin{proof}
Observe that $EN_{G \cup \{x\}}[e] = EN_{G \cup \{x\}}[x],$ so by Proposition~\ref{prop:closed_neigh} we have
$$\match(G \cup \{x\}) \simeq \match(G) \vee \Sigma \match((G \cup \{x\}) \setminus EN_{G \cup \{x\}}[x]).$$ 
Then we have $(G \cup \{x\}) \setminus EN_{G \cup \{x\}} [x] = G \setminus EN_{G}[e],$ and the result follows.
 \end{proof}

\section{General polygonal line tilings}

A (general) polygonal line tiling is a graph formed by a string of cycles, 
each intersecting the previous at an edge, no three intersecting.  
To maintain the last property, we assume all the cycles are of 
length at least four.
We want to prove that the matching complex of such a graph is contractible
or homotopy equivalent to a wedge of spheres.  
Our methods will require that we expand the class of graphs slightly by
allowing two paths attached to adjacent vertices of the final cycle in the
string.  Here is the formal definition.

\begin{definition}
Let $n$ be a positive integer, $k$ and $\ell$ nonnegative integers, and
$s_1, s_2, \ldots, s_n$ be a sequence of integers satisfying 
$s_j\ge 4$ for all $j$. 
Let ${\cal G}^{k,\ell}_{s_1,s_2,\ldots, s_n}$ be the set of graphs 
obtained as follows.
\begin{itemize}
\item For each $i$,  $1\le i\le n-1$, $C_i$ is an $s_i$-cycle containing two disjoint 
      edges $a_{i-1}b_{i-1}$ and $c_id_i$. $C_n$ is an $s_n$-cycle containing two
      disjoint edges $a_{n-1}b_{n-1}$ and $a_nb_n$. 
\item $T$ is a length $k$ path  on vertices $t_0,t_1,\ldots, t_k$ (if $k\ge 1$) and 
      $U$ is a length $\ell$ path on vertices $u_0,u_1,\ldots, u_\ell$ (if $\ell\ge 1$).
\item The following pairs of vertices are identified: $a_i=c_i$ and $b_i=d_i$
      for all $i$, $1\le i\le n-1$, and $a_n=t_0$, $b_n=u_0$ (if the latter
      vertices exist).
\end{itemize}
Any graph in 
${\cal G}^{k,\ell}_{s_1,s_2,\ldots, s_n}$ is called an {\em (extended) 
polygonal line tiling}. See Figure~\ref{extended_polygonal}.
\end{definition}

Note: This set can contain different graphs with the same $s_i$, $k$ and $\ell$.  This is not important for our arguments. 

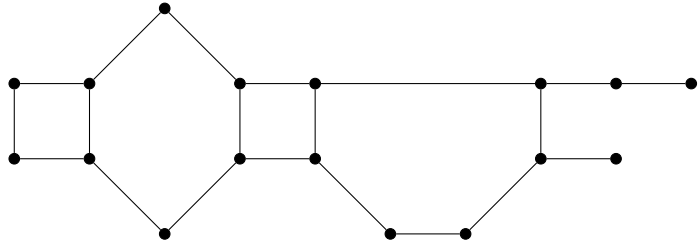
\begin{figure}[h]
\begin{center}
\begin{tikzpicture}
\node[vertex] (a) at (0,2) {};
\node[vertex] (b) at (1,2) {};
\node[vertex] (c) at (2,3) {};
\node[vertex] (d) at (3,2) {};
\node[vertex] (e) at (4,2) {};
\node[vertex] (f) at (7,2) {};
\node[vertex] (g) at (8,2) {};
\node[vertex] (h) at (9,2) {};
\node[vertex] (i) at (8,1) {};
\node[vertex] (j) at (7,1) {};
\node[vertex] (k) at (6,0) {};
\node[vertex] (l) at (5,0) {};
\node[vertex] (m) at (4,1) {};
\node[vertex] (n) at (3,1) {};
\node[vertex] (o) at (2,0) {};
\node[vertex] (p) at (1,1) {};
\node[vertex] (q) at (0,1) {};
\draw (a)--(b)--(c)--(d)--(e)--(f)--(g)--(h);
\draw (i)--(j)--(k)--(l)--(m)--(n)--(o)--(p)--(q);
\draw (a)--(q);
\draw (b)--(p);
\draw (d)--(n);
\draw (e)--(m);
\draw (f)--(j);
\end{tikzpicture}
\caption{Example of extended polygonal line tiling
         in ${\cal G}_{4,6,4,6}^{2,1}$} \label{extended_polygonal}
\end{center}
\end{figure}

\begin{theorem}
If $G$ is any graph in ${\cal G}^{k,\ell}_{s_1,s_2,\ldots, s_n}$, then 
the matching complex of $G$ is contractible or homotopy equivalent to a 
wedge of spheres.
\end{theorem}

\begin{proof}
The proof is by induction on $n$, with an internal induction on $k$ and~$\ell$.

Base case.  Suppose $n=1$.  By symmetry, we can assume $k\ge \ell$.
If $k=\ell=0$, then 
${\cal G}^{0,0}_{s_1}$ is simply the $s_1$-cycle, which is known to 
have matching complex homotopy equivalent to a sphere or a wedge of two
spheres (Kozlov, \cite{kozlov-directed}).

If $k=1$, $0\le \ell\le 1$, 
let $e=t_0t_1=a_1t_1$ and let $h$ be the edge of the
$s_1$-cycle, $h=a_1b_1$.  Thus $EN[e]\subset EN[h]$, and  
by Proposition~\ref{prop:closed_neigh},
\begin{eqnarray*}
\match(G) &\simeq& \match(G  \setminus\{h \}) \vee 
\Sigma\match(G \setminus EN[h]) \\
&=& \match(P_{s_1+\ell+1}) \vee \Sigma\match(P_{s_1-2})
\end{eqnarray*}
where $P_n$ denotes a path on $n$ vertices.

It is known that the matching complex of a path is contractible
or homotopy equivalent to a sphere \cite{kozlov-directed}, so in this case
$\match(G)$ is contractible or homotopy equivalent to a wedge of spheres.

If $k=2$, $0\le \ell\le 2$, let $e=t_1t_2$ and let $h=t_0t_1=a_1t_1$.
Thus $EN[e]\subset EN[h]$, and  
by Proposition~\ref{prop:closed_neigh}
\begin{eqnarray*}
\match(G) &\simeq& \match(G  \setminus\{h \}) \vee 
\Sigma\match(G \setminus EN[h]) \\
&=& \match(H\sqcup P_2)) \vee \Sigma\match(P_{s_1-1+\ell})
\end{eqnarray*}
where $H$ is the $s_1$-cycle with an $\ell$-path attached to the cycle
at an endpoint.  Since $\match(P_2)$ is a single vertex, and the 
matching complex of a disjoint union is the join of the matching
complexes of the components,  $\match(H\sqcup P_2)$ is contractible.
So $\match(G)$ is contractible or homotopy equivalent to a wedge of spheres.

Now, inductively, assume that if both $k$ and $\ell$ are at most $m\ge 2$, then
the matching complex of 
$G\in{\cal G}^{k,\ell}_{s_1}$ is homotopy equivalent to
a wedge of spheres.
Let $\ell\le k = m+1$ and let $G\in{\cal G}^{m+1,\ell}_{s_1}$.  
Let $e=t_mt_{m+1}$ and $h=t_{m-1}t_m$.  
Thus $EN[e]\subset EN[h]$, and  
by Proposition~\ref{prop:closed_neigh},
\begin{eqnarray*}
\match(G) &\simeq& \match(G  \setminus\{h \}) \vee 
\Sigma\match(G \setminus EN[h]) \\
&=& \match(H\sqcup P_2)) \vee \Sigma\match(J)
\end{eqnarray*}
where $H$ is the disjoint union of a graph in 
${\cal G}^{m-1,\ell}_{s_1}$ 
and $P_2$, and hence has contractible
matching complex, and $J\in {\cal G}^{m-2,\ell}_{s_1}$.
If $\ell\le m$, then by the induction assumption, $\match(J)$ is
homotopy equivalent to a wedge of spheres.
If $\ell = m+1$, we repeat the argument with the roles of $k$ and $\ell$
reversed, and reduce to the suspension of a matching complex for a graph
in ${\cal G}^{m-2,m-2}_{s_1}$.  So, again, by induction 
the matching complex of $G$ is 
homotopy equivalent to a wedge of spheres.

This completes the base case, $n=1$.  

Now we assume the result for extended polygonal line tilings with fewer
than $n$ basic cycles, $n\ge 2$, and let 
$G\in {\cal G}^{k,\ell}_{s_1,s_2,\ldots, s_n}$.
As in the base case, we consider different values of $k$ and $\ell$, assuming
$k\ge \ell$.

Assume $k=\ell=0$; then we need to consider separate cases based on the 
size of $s_n$.

Consider $s_n=4$.  Let $e=a_nb_n$ and $h=a_{n-1}b_{n-1}$.  Then $EN(e)\subset EN(h)$,
by Proposition~\ref{prop:open_neighborhood}, 
$\match(G) \simeq \match(G  \setminus\{h \})$.  The graph 
$G\setminus\{h\}$ is in the set 
${\cal G}^{0,0}_{s_1,s_2,\ldots, s_{n-2},s_{n-1}+2}$, so by the induction
assumption, the matching complex of $G$ is 
homotopy equivalent to a wedge of spheres.

Now consider $s_n=5$. 
 The 5-cycle minus the edge $a_{n-1}b_{n-1}$ forms a path of length four
with internal vertices of degree 2.  Then, 
by Proposition~\ref{prop:shorten_path},
$\match(G)$ is homotopy equivalent to the suspension of the matching complex 
of the (multi)graph $H$ 
obtained by shrinking the 5-cycle to a 2-cycle (pair of parallel edges).   Let $e$ and $h$ be those parallel edges in $H$. Then $EN[e] = EN[h]$, and
by Proposition~\ref{prop:shorten_path} and Proposition~\ref{double_edge} we obtain: 
\begin{eqnarray*} \match(G) &\simeq& \Sigma\match(H) \simeq 
\Sigma\left(\match(H\setminus\{h \}) \vee 
\Sigma\match(H \setminus EN[h])\right)\\ 
&\simeq& 
\Sigma\left(\match(H\setminus\{h \}) \right) \vee 
\Sigma^2\match(H \setminus EN[h]).\end{eqnarray*}
Here $H\setminus \{h\}\in {\cal G}^{0,0}_{s_1,s_2,\ldots, s_{n-1}}$
and $H\setminus EN[h]\in 
{\cal G}^{k',\ell'}_{s_1,s_2,\ldots, s_{n-2}}$ for some $k'$ and $\ell'$ with
$k'+\ell'= s_{n-1}-4$. 
(If $n=2$, $H\setminus EN[h]$ is a path of length $s_{n-1}-3$.) By the induction assumption, the matching complex of
each is homotopy equivalent to a wedge of spheres, and hence so is
the matching complex of $G$.

We now consider $s_n=6$.  
 The 6-cycle minus the edge $a_{n-1}b_{n-1}$ contains a path of length four
with internal vertices of degree 2.  So 
by Proposition~\ref{prop:shorten_path}
$\match(G)$ is the homotopy equivalent to the suspension of the matching 
complex of the graph 
$H$ obtained by shrinking the 6-cycle to a 3-cycle.  
Let $h=a_{n-1}b_{n-1}$ and let $e$ be one of the other edges of the 3-cycle 
in $H$.
Then $EN[e]\subset EN[h]$, and 
by Proposition~\ref{prop:closed_neigh}, 
$ \match(G) \simeq \Sigma\match(H)  \simeq 
\Sigma\left(\match(H\setminus\{h \}) \right) \vee 
\Sigma^2\match(H \setminus EN[h]).$
Here $H\setminus \{h\}\in 
{\cal G}^{0,0}_{s_1,s_2,\ldots, s_{n-1}+1}$
and $H\setminus EN[h]\in 
{\cal G}^{k',\ell'}_{s_1,s_2,\ldots, s_{n-2}}$ for some $k'$ and $\ell'$ with
$k'+\ell'= s_{n-1}-4$.  (Again, if $n=2$ $H\setminus EN[h]$ is a path of length $s_{n-1}-3$.)  
By the induction assumption, the matching complex of each is contractible or 
homotopy equivalent to a wedge of spheres, and hence so is
the matching complex of $G$.

Finally, consider $s_n\ge 7$.
 The $s_n$-cycle minus the edge $a_{n-1}b_{n-1}$ contains a path of length four
with internal vertices of degree 2.  Then 
by Proposition~\ref{prop:shorten_path}
$\match(G)$ is the homotopy equivalent to the suspension of the matching complex
of the graph $H$ obtained by shrinking the $s_n$-cycle to an $(s_n-3)$-cycle.  
This process can be repeated until the cycle shrinks to a cycle of length at 
most 6.
So by the above cases, the matching complex of $G$ is 
contractible or homotopy equivalent to a wedge of spheres.

This completes the $k=\ell=0$ case for $n$.

Now assume $k=1$, $0\le \ell\le 1$.  Let 
$e=t_0t_1=a_nt_1$ and let $h$ be the edge of the
$s_n$-cycle $h=a_nb_n$.  Thus $EN[e]\subset EN[h]$, and  
by Proposition~\ref{prop:closed_neigh},
$\match(G) \simeq \match(G  \setminus\{h \}) \vee 
\Sigma\match(G \setminus EN[h])$.
Here $G\setminus\{h\}\in
{\cal G}^{k',\ell'}_{s_1,s_2,\ldots, s_{n-1}}$,
with $k'+\ell' = s_n-1+\ell$, and 
$G\setminus EN[h]\in  
{\cal G}^{k'',\ell''}_{s_1,s_2,\ldots, s_{n-1}}$,
with $k''+\ell'' = s_n-4$.
By the induction assumption, the matching complex of each is contractible or 
homotopy equivalent to a wedge of spheres, and hence so is
the matching complex of $G$.

Next assume $k=2$, $0\le \ell \le 2$.
(For an example see Figure~\ref{reduction}.)
Let $e=t_1t_2$ and let $h=t_0t_1=a_nt_1$.
Thus $EN[e]\subset EN[h]$, and  
by Proposition~\ref{prop:closed_neigh} 
$$
\match(G) \simeq \match(G  \setminus\{h \}) \vee 
\Sigma\match(G \setminus EN[h]).$$

Here $G\setminus \{h\}=H\sqcup P_2$,
where $H\in {\cal G}^{0,\ell}_{s_1,s_2,\ldots, s_n}$.
Since $\match(P_2)$ is a single vertex, and the 
matching complex of a disjoint union is the join of the matching
complexes of the components,  $\match(H\sqcup P_2)$ is contractible.
Also, $G\setminus EN[h]\in
{\cal G}^{k',\ell'}_{s_1,s_2,\ldots, s_{n-1}}$,
with $k'+\ell'=s_n-3+\ell$.
By the induction assumption, the matching complex of each is contractible or 
homotopy equivalent to a wedge of spheres, and hence so is
the matching complex of $G$.

\begin{figure}[h]
\begin{center}
\begin{tikzpicture}
\filldraw[black] (-2.5,.5) circle (1.5pt);
\filldraw[black] (-1.5,.5) circle (1.5pt);
\filldraw[black] (-0.5,.5) circle (1.5pt);
\node[vertex] (a) at (0,1) {};
\node[vertex] (b) at (1,1) {};
\node[vertex] (c) at (2,2) {};
\node[vertex] (d) at (3,1) {};
\node[vertex] (e) at (4,1) {};
\node[vertex] (f) at (5,1) {};
\node[vertex] (g) at (0,0) {};
\node[vertex] (h) at (1,0) {};
\node[vertex] (i) at (3,0) {};
\node[vertex] (j) at (4,0) {};
\draw (a)--(b)--(c)--(d);
\draw (d)--(e) node[draw=none,fill=none,font=\small,midway,above] {$h$};
\draw (e)--(f) node[draw=none,fill=none,font=\small,midway,above] {$e$};
\draw (g)--(h)--(i)--(j);
\draw (a)--(g);
\draw (b)--(h);
\draw (d)--(i);
\draw (7,1) node[below] {$G\in {\cal G}_{s_1,\ldots, 4, 5}^{2,1}$};
\end{tikzpicture}

\vspace*{12pt}

\begin{tikzpicture}
\filldraw[black] (-2.5,.5) circle (1.5pt);
\filldraw[black] (-1.5,.5) circle (1.5pt);
\filldraw[black] (-0.5,.5) circle (1.5pt);
\node[vertex] (a) at (0,1) {};
\node[vertex] (b) at (1,1) {};
\node[vertex] (c) at (2,2) {};
\node[vertex] (d) at (3,1) {};
\node[vertex] (e) at (4,1) {};
\node[vertex] (f) at (5,1) {};
\node[vertex] (g) at (0,0) {};
\node[vertex] (h) at (1,0) {};
\node[vertex] (i) at (3,0) {};
\node[vertex] (j) at (4,0) {};
\draw (a)--(b)--(c)--(d);
\draw (e)--(f);
\draw (g)--(h)--(i)--(j);
\draw (a)--(g);
\draw (b)--(h);
\draw (d)--(i);
\draw (7,0) node[below] {$G\setminus \{h\} = H \sqcup P_1$,
$H \in {\cal G}_{s_1,\ldots, 4, 5}^{0,1}$};
\end{tikzpicture}

\vspace*{12pt}

\begin{tikzpicture}
\filldraw[black] (-2.5,.5) circle (1.5pt);
\filldraw[black] (-1.5,.5) circle (1.5pt);
\filldraw[black] (-0.5,.5) circle (1.5pt);
\node[vertex] (a) at (0,1) {};
\node[vertex] (b) at (1,1) {};
\node[vertex] (c) at (2,2) {};
\node[vertex] (g) at (0,0) {};
\node[vertex] (h) at (1,0) {};
\node[vertex] (i) at (3,0) {};
\node[vertex] (j) at (4,0) {};
\draw (a)--(b)--(c);
\draw (g)--(h)--(i)--(j);
\draw (a)--(g);
\draw (b)--(h);
\draw (7,1) node[below] {$G\setminus EN[h]
 \in {\cal G}_{s_1,\ldots, 4}^{1,2}$};
\end{tikzpicture}
\caption{Example of reduction for $s_n=5$, $k=2$ \label{reduction}}
\end{center}
\end{figure}
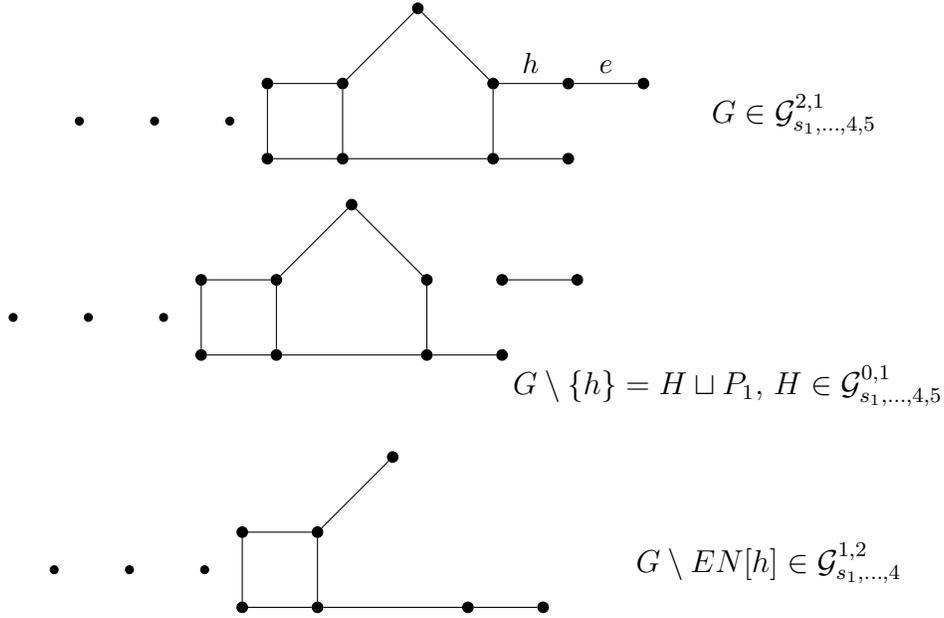

Now, inductively, assume that if both $k$ and $\ell$ are at most $m\ge 2$, then
the matching complex of 
$G\in{\cal G}^{k,\ell}_{s_1,s_2,\ldots, s_n}$
is contractible or homotopy equivalent to a wedge of spheres.
Let $\ell\le k = m+1$ and let 
$G\in{\cal G}^{m+1,\ell}_{s_1,s_2,\ldots, s_n}$
Let $e=t_mt_{m+1}$ and $h=t_{m-1}t_m$.  
Thus $EN[e]\subset EN[h]$, and  
by Proposition~\ref{prop:closed_neigh},
$$ \match(G) \simeq \match(G  \setminus\{h \}) \vee 
\Sigma\match(G \setminus EN[h]).$$
Here $G\setminus\{h\}$ is the disjoint union of a graph in 
${\cal G}^{m-1,\ell}_{s_1,s_2,\ldots, s_n}$
and
$P_2$, and hence has contractible
matching complex, and 
$\Sigma\match(G \setminus EN[h])\in 
{\cal G}^{m-2,\ell}_{s_1,s_2,\ldots, s_n}$.
If $\ell\le m$, then by the induction assumption, 
$\match(G \setminus EN[h])$ is
homotopy equivalent to a wedge of spheres.
If $\ell = m+1$, we repeat the argument with the roles of $k$ and $\ell$
reversed, and reduce to the suspension of a matching complex for a graph
${\cal G}^{m-2,m-2}_{s_1,s_2,\ldots, s_n}$.
So, again, by induction 
the matching complex of $G$ is contractible or 
homotopy equivalent to a wedge of spheres.

This completes the induction on $n$ and hence the proof.

\end{proof}

\section{Line tilings by triangles}

In the last section, we restricted the cycles in the tilings to be of length
four or greater, to avoid three cycles intersecting at a point.
Now we look at the special case of a line of triangles.

\begin{definition}
Let $t$ be a positive integer.
A {\em regular triangular line tiling} is a graph
$P_{3,t}$ with vertex set $V$ and edge set $E$ as follows: 
\begin{itemize}
\item $V = \{a_i\mid 0\le i\le \lceil t/2\rceil \} \cup
          \{b_i\mid 0\le i\le \lfloor t/2\rfloor \} $
\item $E = \{a_ia_{i+1}\mid 0\le i\le \lfloor (t-1)/2\rfloor \} \cup
            \{b_ib_{i+1}\mid 0\le i\le \lfloor (t-2)/2\rfloor \} \cup
            \{a_{i+1}b_i \mid 0\le i\le \lfloor (t-1)/2\rfloor \} \cup \{a_ib_i \mid 0\le i\le \lfloor t/2\rfloor \}.$
\end{itemize}
We extend this definition to $t=0$, where it gives a single edge $a_0b_0$.
\end{definition}

See Figure~\ref{P35}.

\begin{figure}
\begin{center}
\begin{tikzpicture}
\draw (0,3) node[above left] {$a_0$}--
      (2,3) node[above] {$a_1$}--
      (4,3) node[above] {$a_2$}--
      (6,3) node[above right] {$a_3$}--
      (4,0) node[below right] {$b_2$}--
      (2,0) node[below] {$b_1$}--
      (0,0) node[below left] {$b_0$}--cycle;
\draw (0,0)--(2,3)--(2,0)--(4,3)--(4,0);
\filldraw[black] (0,0) circle (2pt);
\filldraw[black] (0,3) circle (2pt);
\filldraw[black] (2,0) circle (2pt);
\filldraw[black] (2,3) circle (2pt);
\filldraw[black] (4,0) circle (2pt);
\filldraw[black] (4,3) circle (2pt);
\filldraw[black] (6,3) circle (2pt);
\end{tikzpicture}
\caption{$P_{3,5}$\label{P35}}
\end{center}
\end{figure}

\begin{theorem}\label{3-gonal_tiling}
Let $P_{3,t}$ be a regular triangular line tiling. Then,
\begin{eqnarray*}
\match(P_{3,0}) &\simeq& *, \quad \match(P_{3,1}) \simeq \match(P_{3, 2}) 
\simeq \bigvee\limits_2 S^0 \\
\match(P_{3, 3 })&\simeq& S^1, \quad 
\match(P_{3,4}) \simeq \bigvee\limits_5 S^1 \quad
\end{eqnarray*}
and for $t\ge 5$,
\[
\match(P_{3,t}) \simeq \Sigma \match(P_{3,t-3}) \vee \Sigma \match(P_{3,t-3}) \vee \Sigma^2\match(P_{3,t-5}). 
\]

Thus  $\match(P_{3,t})$ is contractible or homotopy equivalent to a wedge of 
spheres for all $t \geq 1.$
\end{theorem}

\begin{proof}
The homotopy types for $t \leq 4$ are straightforward. See Appendix. Now assume $t\ge 5$.

Since $EN(a_0b_0) \subset EN(a_1b_1)$ in $P_{3,t}$, Proposition~\ref{prop:open_neighborhood} gives us $\match(P_{3,t}) \simeq \match(P_{3,t} \setminus \{a_1b_1\})$. 

In $\match(P_{3,t} \setminus \{a_1b_1\})$ we see  $EN[a_0b_0]\subset EN[a_1 b_0]$.  Hence by Proposition~\ref{prop:closed_neigh},
\[
\match(P_{3,t} \setminus \{a_1b_1\}) \simeq \match(P_{3,t} \setminus \{a_1b_1,a_1b_0\}) \vee \Sigma \match((P_{3,t} \setminus \{a_1b_1\}) \setminus EN[a_1b_0]).
\]
Since $(P_{3,t} \setminus \{a_1b_1\}) \setminus EN[a_1b_0]$ is isomorphic to $P_{3,t-3}$,
\[
\match(P_{3,t} \setminus \{a_1b_1\}) \simeq \match(P_{3,t} \setminus \{a_1b_1,a_1b_0\}) \vee \Sigma \match(P_{3,t-3}).
\]
We now turn our attention to 
$\match(P_{3,t} \setminus \{a_1b_1,a_1b_0\})$. 
 In $P_{3,t} \setminus \{a_1b_1,a_1b_0\}$ the vertices
$a_2, a_1, a_0, b_0,$ and  $b_1$ form an induced path of length $4$, call it $X$. 
 Contracting path $X$ we obtain a graph isomorphic to $P_{3,t-3}$ with an additional double edge $x$ (with same vertices as edge $a_2b_1$); call this graph $P_{3,t-3}'= P_{3,t-3} \cup \{x\}$.
 Then by Proposition~\ref{prop:shorten_path}, 
 $$\match(P_{3,t} \setminus \{a_1b_1,a_1b_0\}) \simeq \Sigma\match(P_{3,t-3}').$$
 Further, we apply Proposition~\ref{double_edge} and obtain 
 $$\match(P_{3,t-3}') \simeq \match(P_{3,t-3}) \vee \Sigma \match(P_{3,t-3} \setminus EN_{P_{3,t-3}}[x]).$$ 
 The graph $P_{3,t-3} \setminus EN_{P_{3,t-3}}[x]$ is isomorphic to $P_{3,t-5}$. 
 Together the homotopy equivalences obtained imply

 \[
\match(P_{3,t}) \simeq \Sigma \match(P_{3,t-3}) \vee \Sigma \match(P_{3,t-3}) \vee \Sigma^2\match(P_{3,t-5}). 
\]
 Hence $\match(P_{3,t})$ is homotopy equivalent to a wedge of spheres, or contractible for all $t \geq 1.$ 
\end{proof}

\begin{corollary}
Let $s(t,d)$ be the number of spheres of dimension $d$ in the wedge that is the
homotopy type of $\match(P_{3,t})$.
Then for $t\ge 7$ and $d\ge 2$, $s(t,d) = {2s(t-3,d-1)}+s(t-5,d-2)$ and
$$\sum_{t\ge 2, d\ge 0}
s(t,d)x^ty^d = \frac{2x^2+x^3y+5x^4y+2x^6y^2}{1-2x^3y-x^5y^2}.$$
\end{corollary}

\begin{proof}
Let $\displaystyle q(x,y) =
\sum_{\genfrac{}{}{0pt}{2}{t\ge 2}{d\ge 0}} s(t,d)x^t y^d$.  From Theorem~\ref{3-gonal_tiling}, we get $s(t,d)$ for $t\le 6$ or $d\le 2$, and 
the recursion 
$\match(P_{3,t}) \simeq
\bigvee\limits_2\Sigma \match(P_{3,t-3}) \vee \Sigma^2\match(P_{3,t-5})$ gives $s(t,d)=2s(t-3,d-1)+s(t-5,d-2)$ for $t\ge 7$, $d\ge 2$. 

Thus

\begin{eqnarray*}
q(x,y) &=& 2x^2+x^3y+5x^4y+4x^5y+4x^6y^2 + \sum_{\genfrac{}{}{0pt}{2}{t\ge 7}{d\ge 2}} s(t,d)x^ty^d\\
       &=& 2x^2+x^3y+5x^4y+4x^5y+4x^6y^2 \\ &+&
           \sum_{\genfrac{}{}{0pt}{2}{t\ge 7}{d\ge 2}} 2s(t-3,d-1)x^ty^d +
           \sum_{\genfrac{}{}{0pt}{2}{t\ge 7}{d\ge 2}} s(t-5,d-2)x^ty^d \\
       &=& 2x^2+x^3y+5x^4y+4x^5y+4x^6y^2  \\ &+&
           \sum_{\genfrac{}{}{0pt}{2}{t\ge 4}{d\ge 1}} 2s(t,d)x^{t+3}y^{d+1} +
           \sum_{\genfrac{}{}{0pt}{2}{t\ge 2}{d\ge 0}} s(t,d)x^{t+5}y^{d+2} \\
       &=& 2x^2+x^3y+5x^4y+4x^5y+4x^6y^2  \\ &+&
           \sum_{\genfrac{}{}{0pt}{2}{t\ge 2}{d\ge 0}} 2s(t,d)x^{t+3}y^{d+1} - 4x^5y-2x^6y^2 +
           \sum_{\genfrac{}{}{0pt}{2}{t\ge 2}{d\ge 0}} s(t,d)x^{t+5}y^{d+2}.
\end{eqnarray*}
So
$$ q(x,y)(1-2x^3y-x^5y^2) = 2x^2+x^3y+5x^4y+2x^6y^2.$$
\end{proof}
\begin{theorem}
For $t\ge 2$, let
$D_t$ be the set of dimensions of the spheres occurring in the wedge of spheres
that gives the homotopy type of $\match(P_{3,t})$.
Let
$\displaystyle I_t = \left[ \left\lfloor \frac{t}{3} \right\rfloor, \frac{2t-f(t)}{5}\right]$,
where
$$f(t) = \left\{ \begin{array}{cl}
         5 & \mbox{if $t \equiv 0 \!\!\! \pmod{5}$}\\
         2 & \mbox{if $t \equiv 1 \!\!\! \pmod{5}$}\\
         4 & \mbox{if $t \equiv 2 \!\!\! \pmod{5}$}\\
         1 & \mbox{if $t \equiv 3 \!\!\! \pmod{5}$}\\
         3 & \mbox{if $t \equiv 4 \!\!\! \pmod{5}$} \end{array}\right. .$$
Then $D_t=I_t$.
\end{theorem}
\begin{proof}
Part I.
$\displaystyle D_t \subseteq I_t$.

The proof is by induction on $t\ge 2$.  The statement is true for the base cases,
$2\le t\le 6$.  So assume $t\ge 7$ and the statement is  true for all smaller $t$.
Theorem~\ref{3-gonal_tiling} implies that
$D_t = \{r+1\mid r\in D_{t-3}\} \cup \{r+2\mid r\in D_{t-5}\}$.
We consider first the smallest integer in $D_t$.

\begin{eqnarray*}
\min(D_t)&=& \min(\min(D_{t-3})+1, \min(D_{t-5})+2) \\
           &=&\min\left(\left\lfloor \frac{t-3}{3}\right\rfloor+1,\left\lfloor\frac{t-5}{3}\right\rfloor+2\right)\\
           &=&\min\left(\left\lfloor \frac{t}{3}\right\rfloor, \left\lfloor \frac{t+1}{3}\right\rfloor\right)
           =\left\lfloor \frac{t}{3}\right\rfloor \end{eqnarray*}

Now for the largest integer in $D_t$.
\begin{eqnarray*}
\max(D_t) &=& \max(\max(D_{t-3})+1,\max(D_{t-5})+2)\\
          &=& \max\left(\frac{2t-1-f(t-3)}{5},\frac{2t-f(t-5)}{5}\right)
\end{eqnarray*}

We calculate this for each congruence class modulo 5.
\begin{itemize}
\item $t \equiv t-5 \equiv 0  \pmod{5}$, $t-3 \equiv 2  \pmod{5}$

      $\displaystyle \max(D_t) = \max \left(\frac{2t-5}{5},\frac{2t-5}{5}\right) =
       \frac{2t-5}{5} = \frac{2t-f(t)}{5}$
\item $t \equiv t-5 \equiv 1  \pmod{5}$, $t-3 \equiv 3  \pmod{5}$

      $\displaystyle \max(D_t) = \max \left(\frac{2t-2}{5},\frac{2t-2}{5}\right) =
       \frac{2t-2}{5} = \frac{2t-f(t)}{5}$
\item $t \equiv t-5 \equiv 2  \pmod{5}$, $t-3 \equiv 4  \pmod{5}$

      $\displaystyle \max(D_t) = \max \left(\frac{2t-4}{5},\frac{2t-4}{5}\right) =
       \frac{2t-4}{5} = \frac{2t-f(t)}{5}$
\item $t \equiv t-5 \equiv 3  \pmod{5}$, $t-3 \equiv 0  \pmod{5}$

      $\displaystyle \max(D_t) = \max \left(\frac{2t-6}{5},\frac{2t-1}{5}\right) =
       \frac{2t-1}{5} = \frac{2t-f(t)}{5}$
\item $t \equiv t-5 \equiv 4  \pmod{5}$, $t-3 \equiv 1  \pmod{5}$

      $\displaystyle \max(D_t) = \max \left(\frac{2t-3}{5},\frac{2t-3}{5}\right) =
       \frac{2t-3}{5} = \frac{2t-f(t)}{5}$
\end{itemize}
So $\min(D_t) = \min(I_t)$ and $\max(D_t)=\max(I_t)$.

Part II.
$\displaystyle I_t \subseteq D_t$.

Note that in the expansion of the rational function for $q(x,y)$ there is
no subtraction.
We consider one set of monomials that occur in $q(x,y)$ with positive
coefficients, namely those  of the form $(2x^2)(2x^3y)^\alpha (x^5y^2)^\beta
= 2^{\alpha+1}x^{3\alpha+5\beta+2}y^{\alpha+2\beta}$.
Each such monomial $cx^ty^d$ represents $c$ spheres of dimension $d$ in the
wedge of spheres for the  homotopy type of $\match(P_{3,t})$, that is, an element
$d$ of $D_t$.
We are not concerned with the coefficient $c$, which for these monomials
is positive.  So we will consider just the exponents.  We know that
for every nonnegative integers $\alpha$ and $\beta$,
$\alpha+2\beta\in D_{3\alpha+5\beta+2}$.
Also, for every integer $t\ge 5$, there exist $\alpha\ge 0$ and $\beta\ge 0$
such that $t= 3\alpha+5\beta+2$. (In general,
the $\alpha$ and $\beta$ are not uniquely determined.)
In what follows we sometimes assume $t\ge 14$;
it is easy to check that the statement of the theorem is true for smaller $t$.  (See Appendix.)
We show that for fixed $t\ge 14$, all of these elements of $D_t$ fill the
interval $I_t$.

Let $A_t=\{(\alpha,\beta)\mid 3\alpha+5\beta+2 = t\}$
and $f(\alpha,\beta) = \alpha+2\beta$.
Note that if $(\alpha,\beta)\in A_t$ and $\alpha\ge 5$, then
$(\alpha-5,\beta+3)\in A_t$ and
$f(\alpha-5,\beta+3)= f(\alpha,\beta)+1$.
Using this we will produce an interval of sphere dimensions for fixed $t$.

Fix $t\ge 14$.
The minimum of $\alpha+2\beta$ for
$(\alpha,\beta)\in A_t$ occurs when $\alpha$ is greatest and $\beta$ is
least; values are in the following table. 

\vspace*{6pt}

\begin{tabular}{|l|l|l|}\hline
$t\equiv 0 \pmod{3}$ & $\alpha = (t-12)/3$, $\beta = 2$ &
         $\alpha+2\beta = t/3$ \\ \hline
$t\equiv 1 \pmod{3}$ & $\alpha = (t-7)/3$, $\beta = 1$ &
         $\alpha+2\beta = (t-1)/3$ \\ \hline
$t\equiv 2 \pmod{3}$ & $\alpha = (t-2)/3$, $\beta = 0$ &
         $\alpha+2\beta = (t-2)/3$  \\ \hline
\end{tabular}

\vspace*{6pt}

The maximum of $\alpha+2\beta$ for
$(\alpha,\beta)\in A_t$ occurs when $\alpha$ is least and $\beta$ is
greatest; values are in the following table.

\vspace*{6pt}

\begin{tabular}{|l|l|l|}\hline
$t\equiv 0 \pmod{5}$ & $\alpha = 1 $, $\beta = (t-5)/5$ &
         $\alpha+2\beta = (2t-5)/5$ \\ \hline
$t\equiv 1 \pmod{5}$ & $\alpha = 3 $, $\beta = (t-11)/5$ &
         $\alpha+2\beta = (2t-7)/5$ \\ \hline
$t\equiv 2 \pmod{5}$ & $\alpha = 0 $, $\beta = (t-2)/5$ &
         $\alpha+2\beta = (2t-4)/5$ \\ \hline
$t\equiv 3 \pmod{5}$ & $\alpha = 2 $, $\beta = (t-8)/5$ &
         $\alpha+2\beta = (2t-6)/5$ \\ \hline
$t\equiv 4 \pmod{5}$ & $\alpha = 4 $, $\beta = (t-14)/5$ &
         $\alpha+2\beta = (2t-8)/5$ \\ \hline
\end{tabular}

\vspace*{6pt}

Note in all cases a pair in $A_t$ produces the minimum value in the
set  $I_t$, but in some cases no pair in $A_t$ produces
the maximum value in $I_t$.
However, the maximum value
produced by a pair in $A_t$ is at least one less than the maximum
value of $I_t$.  Since the set produced by $A_t$ is itself an interval, missing
at most one element (the top) of the interval $I_t$, and
we know by Part I that the top element of $I_t$ is in $D_t$, we conclude that
$\displaystyle D_t = I_t$.
\end{proof}

\section{Regular pentagonal line tilings}
We consider one more particular case, pentagonal line tilings.

\begin{definition}
Let $t$ be a positive integer.
A {\em regular pentagonal line tiling} is a graph
$P_{5,t}$ with vertex set $V$ and edge set $E$ as follows:
\begin{itemize}
\item $V = \{a_i\mid 0\le i\le \lceil 3t/2\rceil \} \cup
          \{b_i\mid 0\le i\le \lfloor 3t/2\rfloor \} $
\item $E = \{a_ia_{i+1}\mid 0\le i\le \lfloor (3t-1)/2\rfloor \} \cup
            \{b_ib_{i+1}\mid 0\le i\le \lfloor (3t-2)/2\rfloor \} \cup
            \{a_{3j}b_{3j} \mid 0\le j\le \lfloor t/2\rfloor \} \cup
            \{a_{3j+2}b_{3j+1} \mid 0\le j\le \lfloor (t-1)/2\rfloor \}$
\end{itemize}
\end{definition}

See Figure~\ref{P54}.

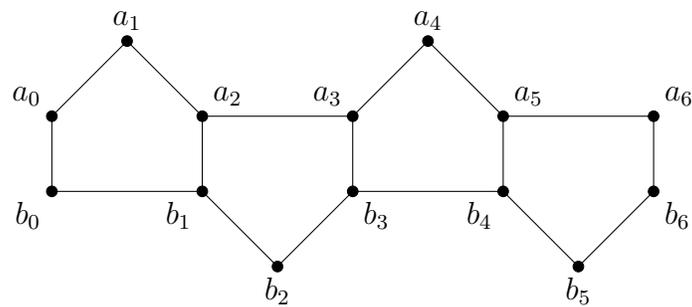
\begin{figure}
\begin{center}
\begin{tikzpicture}
\draw (0,2) node[above left] {$a_0$}--
      (1,3) node[above] {$a_1$}--
      (2,2) node[above right] {$a_2$}--
      (4,2) node[above left] {$a_3$}--
      (5,3) node[above] {$a_4$}--
      (6,2) node[above right] {$a_5$}--
      (8,2) node[above right] {$a_6$}--
      (8,1) node[below right] {$b_6$}--
      (7,0) node[below] {$b_5$}--
      (6,1) node[below left] {$b_4$}--
      (4,1) node[below right] {$b_3$}--
      (3,0) node[below] {$b_2$}--
      (2,1) node[below left] {$b_1$}--
      (0,1) node[below left] {$b_0$}--cycle;
\draw (2,1)--(2,2);
\draw (4,1)--(4,2);
\draw (6,1)--(6,2);
\filldraw[black] (0,1) circle (2pt);
\filldraw[black] (0,2) circle (2pt);
\filldraw[black] (1,3) circle (2pt);
\filldraw[black] (2,1) circle (2pt);
\filldraw[black] (2,2) circle (2pt);
\filldraw[black] (3,0) circle (2pt);
\filldraw[black] (4,1) circle (2pt);
\filldraw[black] (4,2) circle (2pt);
\filldraw[black] (5,3) circle (2pt);
\filldraw[black] (6,1) circle (2pt);
\filldraw[black] (6,2) circle (2pt);
\filldraw[black] (7,0) circle (2pt);
\filldraw[black] (8,1) circle (2pt);
\filldraw[black] (8,2) circle (2pt);
\end{tikzpicture}
\end{center}\caption{Graph $P_{5,4}$}
\label{P54}
\end{figure}

Let $(F_n)$ be the standard Fibonacci sequence, $F_1=F_2=1$, $F_n=F_{n-1}+F_{n-2}$ for $n\ge 3$.

\begin{theorem}\label{thm:5_gon}
Let $P_{5,t}$ be the pentagonal line tiling with $t\ge 1$.  Then
 $\displaystyle \match(P_{5,t})\simeq \bigvee_{F_{t+2}-1} S^t$. 
\end{theorem}
\begin{proof}
The proof is by induction, and we work with two sequences of graphs.  Let $G_t=P_{5,t}$.
Let $H_t$ be the graph obtained by appending one edge to the graph $G_t$, as shown in
Figure~\ref{pentagons-H}.
  We will use Proposition~\ref{prop:shorten_path} and Proposition~\ref{double_edge} to reduce the matching complex of $G_t$ to
the wedge of suspensions of the matching complexes of $G_{t-1}$ and $H_{t-2}$.

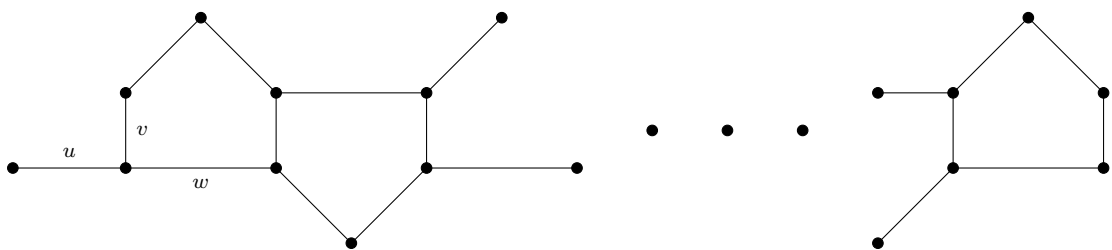
\begin{figure}
\begin{center}
\begin{tikzpicture}
\draw (-1.5,1)--(0,1) node[draw=none,fill=none,font=\scriptsize,midway,above] {$u$};
\draw (0,2)--(1,3)--(2,2)--(2,1);
\draw (0,1)--(2,1) node[draw=none,fill=none,font=\scriptsize,midway,below] {$w$};
\draw (0,1)--(0,2) node[draw=none,fill=none,font=\scriptsize,midway,right] {$v$};
\draw (2,2)--(4,2)--(4,1)--(3,0)--(2,1);
\draw (4,2)--(5,3);
\draw (4,1)--(6,1);
\draw (10,2)--(11,2);
\draw (10,0)--(11,1);
\draw (11,2)--(12,3)--(13,2)--(13,1)--(11,1)--(11,2);
\filldraw[black] (-1.5,1) circle (2pt);
\filldraw[black] (0,1) circle (2pt);
\filldraw[black] (0,2) circle (2pt);
\filldraw[black] (1,3) circle (2pt);
\filldraw[black] (2,1) circle (2pt);
\filldraw[black] (2,2) circle (2pt);
\filldraw[black] (3,0) circle (2pt);
\filldraw[black] (4,1) circle (2pt);
\filldraw[black] (4,2) circle (2pt);
\filldraw[black] (5,3) circle (2pt);
\filldraw[black] (6,1) circle (2pt);
\filldraw[black] (7,1.5) circle (2pt);
\filldraw[black] (8,1.5) circle (2pt);
\filldraw[black] (9,1.5) circle (2pt);
\filldraw[black] (10,0) circle (2pt);
\filldraw[black] (10,2) circle (2pt);
\filldraw[black] (11,1) circle (2pt);
\filldraw[black] (11,2) circle (2pt);
\filldraw[black] (12,3) circle (2pt);
\filldraw[black] (13,1) circle (2pt);
\filldraw[black] (13,2) circle (2pt);
\end{tikzpicture}
\end{center}\caption{Graph $H_t$}
\label{pentagons-H}
\end{figure}

We first find the homotopy type of the matching complex of $H_t$.
It is straightforward  to check that
$\displaystyle \match(H_1)\simeq \bigvee_2 S^1 $. 
Let $u$ be the pendant edge and $v$ and $w$ its neighboring edges, as shown in Figure~\ref{pentagons-H}.
The edge $u$ is a simplicial edge, because its two neighbors are neighbors of each other, so
we can apply Proposition~\ref{prop:simplicial_edge}, and conclude that
$$\match(H_t) \simeq \Sigma\match(H_t \setminus EN[v]) \vee \Sigma\match(H_t\setminus EN[w]).$$
The graph $H_t\setminus EN[v]$ is isomorphic to $H_{t-1}$.
The graph $H_t\setminus EN[w]$ is isomorphic to the graph $H_{t-2}$ with an additional
path of length 3 attached to another vertex of the first pentagon.  
(In the case of $t=2$, $H_t\setminus EN[w]$ is a path of length~5.) 
By Corollary~\ref{path_length_3}, this path can be collapsed to give
$\match(H_t\setminus EN[w])\simeq \Sigma\match(H_{t-2})$. Thus, $\match(H_t)\simeq \Sigma\match(H_{t-1}) \vee \Sigma^2\match(H_{t-2})$.
By induction we conclude that
$\displaystyle \match(H_t)\simeq \bigvee_{F_{t+2}} S^t$.

Now consider $G_t=P_{5,t}$. 
Let $G_t'$ be the multigraph obtained from $G_t$ by duplicating
the first ``vertical'' edge $v$ in $G_t$; denote the duplicate edge $x$.
(See Figure~\ref{pentagons-G'}.)
Proposition~\ref{prop:shorten_path} gives
$\match(G_t) = \Sigma\match(G_{t-1}')$. Since $G_{t-1}' = G_{t-1} \cup \{x\}$, by applying Proposition~\ref{double_edge} we obtain:
$$\match(G_{t-1}') \simeq \match(G_{t-1}) \vee \Sigma \match(G_{t-1} \setminus EN_{G_{t-1}}[v]).$$ 
Further, graph $G_{t-1} \setminus EN_{G_{t-1}}[v])$ is isomorphic to $H_{t-2}$, so  

$$\match(G_t)\simeq \Sigma\match(G_{t-1}') \simeq
\Sigma\match(G_{t-1})\vee\Sigma^2\match(H_{t-2}).$$

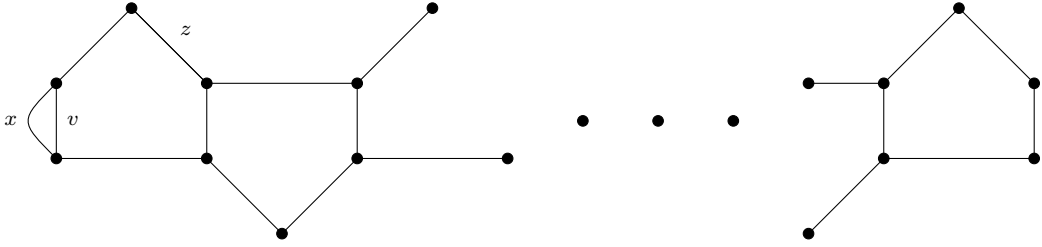
\begin{figure}
\begin{center}
\begin{tikzpicture}
\draw (0,2)--(1,3)--(2,2)--(2,1);
\draw (1,3)--(2,2) node[draw=none,fill=none,font=\scriptsize,midway,above right] {$z$};
\draw (0,1)--(2,1); 
\draw (0,1)--(0,2) node[draw=none,fill=none,font=\scriptsize,midway,right] {$v$};
\draw (0,1) .. controls (-.5,1.5) .. (0,2) node[draw=none,fill=none,font=\scriptsize,midway,left] {$x$};
\draw (2,2)--(4,2)--(4,1)--(3,0)--(2,1);
\draw (4,2)--(5,3);
\draw (4,1)--(6,1);
\draw (10,2)--(11,2);
\draw (10,0)--(11,1);
\draw (11,2)--(12,3)--(13,2)--(13,1)--(11,1)--(11,2);

\filldraw[black] (0,1) circle (2pt);
\filldraw[black] (0,2) circle (2pt);
\filldraw[black] (1,3) circle (2pt);
\filldraw[black] (2,1) circle (2pt);
\filldraw[black] (2,2) circle (2pt);
\filldraw[black] (3,0) circle (2pt);
\filldraw[black] (4,1) circle (2pt);
\filldraw[black] (4,2) circle (2pt);
\filldraw[black] (5,3) circle (2pt);
\filldraw[black] (6,1) circle (2pt);
\filldraw[black] (7,1.5) circle (2pt);
\filldraw[black] (8,1.5) circle (2pt);
\filldraw[black] (9,1.5) circle (2pt);
\filldraw[black] (10,0) circle (2pt);
\filldraw[black] (10,2) circle (2pt);
\filldraw[black] (11,1) circle (2pt);
\filldraw[black] (11,2) circle (2pt);
\filldraw[black] (12,3) circle (2pt);
\filldraw[black] (13,1) circle (2pt);
\filldraw[black] (13,2) circle (2pt);
\end{tikzpicture}
\end{center}
\caption{Graph $G_t'$}
\label{pentagons-G'}
\end{figure}
It is straightforward to check that
$\displaystyle \match(G_1)\simeq  S^1 $ and
$\displaystyle \match(G_2)\simeq \bigvee_2 S^2 $.
By induction we see that $\match(G_t)$ is homotopy equivalent to a wedge of $t$-spheres:
$$\match(G_t)
\simeq
\Sigma\match(G_{t-1})\vee\Sigma^2\match(H_{t-2}) \simeq \Sigma(\bigvee_{F_{t+1}-1} S^{t-1})\vee\Sigma^2(\bigvee_{F_{t}} S^{t-2})
\simeq \bigvee_{F_{t+1}-1 +F_t} S^t.$$
So
 $\displaystyle \match(P_{5,t})\simeq \bigvee_{F_{t+2}-1} S^t$.
 \end{proof}

\section{Conclusion}
Our focus in this paper has been on extending the set of graphs whose matching complexes are known to be contractible or homotopy equivalent to a wedge of spheres.  We are interested in larger classes of graphs, particularly of planar graphs, with this property.  We believe that the methods in this paper will be useful in this regard.

\section{Acknowledgments}
This article is based on work supported by the National Science Foundation [Grant No.\ DMS-144014] while the authors participated in the 2020/2021 Summer Research in Mathematics Program of the Mathematical Sciences Research Institute, Berkeley, California.

M.\ Jeli\'{c} Milutinovi\'{c} has been supported by the Science Fund of Serbia
[Project No.\ 7744592 MEGIC
``Integrability and Extremal Problems in Mechanics, Geometry and
Combinatorics''] and 
by the Ministry of
Education, Science, and Technological Development of the Republic of Serbia.
[Grant No.\ 451-03-68/2022-14/200104 at the Faculty of Mathematics, University of Belgrade].

\pagebreak

\appendix \section{Matching complexes for $P_{3,t}$, $1\le t\le 4$}

$P_{3,1}$: $V=\{a_0,a_1,b_0\}$, $E=\{a_0a_1,a_1b_0,a_0b_0\}$

\begin{center}
\begin{tikzpicture}
\draw (0,3) node[above left] {$a_0$}--
      (2,3) node[above] {$a_1$}--
      (0,0) node[below left] {$b_0$}--cycle;
\filldraw[black] (0,0) circle (2pt);
\filldraw[black] (0,3) circle (2pt);
\filldraw[black] (2,3) circle (2pt);

\draw (1,-1) node[below] {$P_{3,1}$};

\draw (5,1) node[left] {$a_0b_0$};
\filldraw[black] (5,1) circle (2pt);
\draw (6,2.5) node[right] {$a_0a_1$};
\filldraw[black] (6,2.5) circle (2pt);
\draw (6,1) node[right] {$a_1b_0$};
\filldraw[black] (6,1) circle (2pt);
\draw (5.5,-1) node[below] {$\match(P_{3,1})$};
\end{tikzpicture}
\end{center}

$P_{3,2}$: $V=\{a_0,a_1,b_0,b_1\}$, $E=\{a_0a_1,b_0b_1,a_1b_0,a_0b_0,a_1b_1\}$ 

\begin{center}
\begin{tikzpicture}
\draw (0,3) node[above left] {$a_0$}--
      (2,3) node[above] {$a_1$}--
      (0,0) node[below left] {$b_0$}--cycle;
\draw (0,0)--(2,0) node[below right] {$b_1$}--(2,3);
\filldraw[black] (0,0) circle (2pt);
\filldraw[black] (0,3) circle (2pt);
\filldraw[black] (2,3) circle (2pt);
\filldraw[black] (2,0) circle (2pt);

\draw (1,-1) node[below] {$P_{3,2}$};

\draw (5,2) node[above] {$a_0b_0$};
\filldraw[black] (5,2) circle (2pt);
\draw (7,2) node[above] {$a_1b_1$};
\filldraw[black] (7,2) circle (2pt);
\draw (6,1) node[below] {$a_1b_0$};
\filldraw[black] (6,1) circle (2pt);
\draw (8,2.5) node[right] {$a_0a_1$};
\filldraw[black] (8,2.5) circle (2pt);
\draw (8,0.5) node[right] {$b_0b_1$};
\filldraw[black] (8,0.5) circle (2pt);
\draw (7,-1) node[below] {$\match(P_{3,2})$};
\draw (5,2)--(7,2);
\draw (8,2.5)--(8,0.5);
\end{tikzpicture}
\end{center}

$P_{3,3}$: $V=\{a_0,a_1,a_2,b_0,b_1\}$, $E=\{a_0a_1,a_1a_2,b_0b_1,a_1b_0,a_2b_1,a_0b_0,a_1b_1\}$

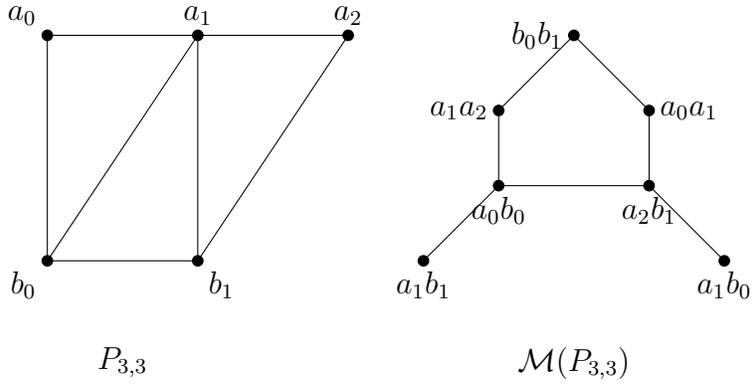
\begin{center}
\begin{tikzpicture}
\draw (0,3) node[above left] {$a_0$}--
      (2,3) node[above] {$a_1$}--
      (0,0) node[below left] {$b_0$}--cycle;
\draw (2,3)--(4,3) node[above] {$a_2$}--(2,0);
\draw (0,0)--(2,0) node[below right] {$b_1$}--(2,3);
\filldraw[black] (0,0) circle (2pt);
\filldraw[black] (0,3) circle (2pt);
\filldraw[black] (2,3) circle (2pt);
\filldraw[black] (2,0) circle (2pt);
\filldraw[black] (4,3) circle (2pt);

\draw (1,-1) node[below] {$P_{3,3}$};

\draw (5,0) node[below] {$a_1b_1$}--
      (6,1) node[below] {$a_0b_0$}--
      (6,2) node[left] {$a_1a_2$}--
      (7,3) node[left] {$b_0b_1$}--
      (8,2) node[right] {$a_0a_1$}--
      (8,1) node[below] {$a_2b_1$}--
      (9,0) node[below] {$a_1b_0$};
\draw (6,1)--(8,1);
\filldraw[black] (5,0) circle (2pt);
\filldraw[black] (6,1) circle (2pt);
\filldraw[black] (6,2) circle (2pt);
\filldraw[black] (7,3) circle (2pt);
\filldraw[black] (8,1) circle (2pt);
\filldraw[black] (8,2) circle (2pt);
\filldraw[black] (9,0) circle (2pt);
\draw (7,-1) node[below] {$\match(P_{3,3})$};
\end{tikzpicture}
\end{center}

$P_{3,4}$: $V=\{a_0,a_1,a_2,b_0,b_1,b_2\}$, 
$E=\{a_0a_1,a_1a_2,b_0b_1,b_1b_2,a_1b_0,a_2b_1,a_0b_0,a_1b_1,a_2b_2\}$, 

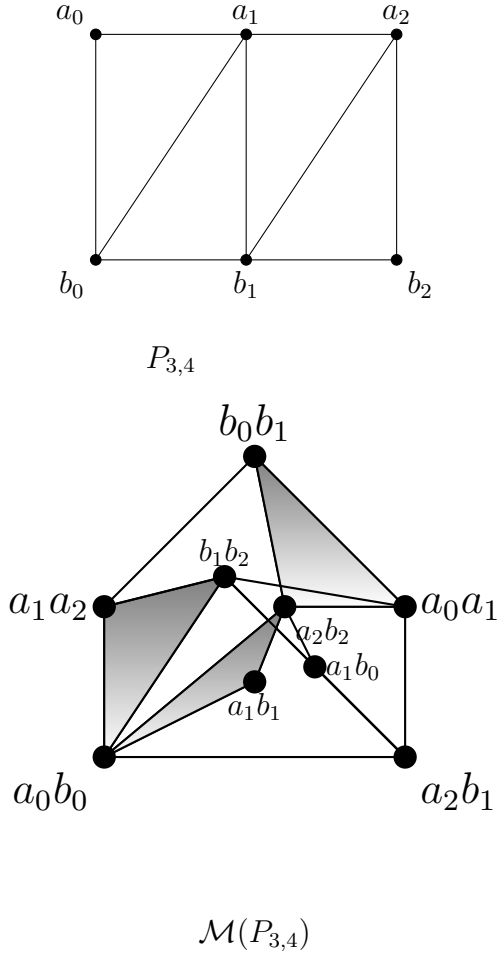
\begin{center}
\begin{tikzpicture}
\draw (0,3) node[above left] {$a_0$}--
      (2,3) node[above] {$a_1$}--
      (0,0) node[below left] {$b_0$}--cycle;
\draw (2,3)--(4,3) node[above] {$a_2$}--(2,0);
\draw (0,0)--(2,0) node[below] {$b_1$}--(2,3);
\draw (4,3)--(4,0) node[below right] {$b_2$}--(2,0);
\filldraw[black] (0,0) circle (2pt);
\filldraw[black] (0,3) circle (2pt);
\filldraw[black] (2,3) circle (2pt);
\filldraw[black] (2,0) circle (2pt);
\filldraw[black] (4,3) circle (2pt);
\filldraw[black] (4,0) circle (2pt);

\draw (1,-1) node[below] {$P_{3,4}$};
\end{tikzpicture}

\scalebox{2}{
\begin{tikzpicture}

\node[scale=.7,below left] at (6,1) {$a_0b_0$};
\node[scale=.7,left] at (6,2) {$a_1a_2$};
\node[scale=.7,above] at (7,3) {$b_0b_1$};
\node[scale=.7,right] at (8,2) {$a_0a_1$};
\node[scale=.7,below right] at (8,1) {$a_2b_1$};
\node[scale=.5,right] at (7.4,1.6) {$a_1b_0$};
\node[scale=.5,below] at (7,1.5) {$a_1b_1$};
\node[scale=.5,above] at (6.8,2.2) {$b_1b_2$};
\node[scale=.5,below right] at (7.2,2) {$a_2b_2$};
\draw (7,1.5)--(6,1)--(6,2)--(7,3)--(8,2)--(8,1)--(7.4,1.6);
\draw (6,1)--(8,1);
\draw (7.4,1.6)--(7.2,2);
\draw (6,1)--(7.2,2);
\draw (8,1)--(7.4,1.6)--(6.8,2.2);
\draw (7,3)--(7.2,2.0)--(7,1.5); 
\draw (6,1)--(6.8,2.2); 
\draw (8,2)--(7.2,2);

\shadedraw (8,2)--(7,3)--(7.2,2)--(8,2);
\shadedraw (6,1)--(6.8,2.2)--(6,2)--(6,1);
\shadedraw (6,1)--(7,1.5)--(7.2,2)--(6,1);

\draw (6,2)--(6.8,2.2)--(8,2);
\draw (8,1)--(7.4,1.6)--(6.8,2.2);

\filldraw[black] (7,1.5) circle (2pt);
\filldraw[black] (6,1) circle (2pt);
\filldraw[black] (6,2) circle (2pt);
\filldraw[black] (7,3) circle (2pt);
\filldraw[black] (8,1) circle (2pt);
\filldraw[black] (8,2) circle (2pt);
\filldraw[black] (7.4,1.6) circle (2pt);
\filldraw[black] (7.2,2.0) circle (2pt);
\filldraw[black] (6.8,2.2) circle (2pt);

\node[scale=.5,below] at (7,0) {$\match(P_{3,4})$};
\end{tikzpicture}}
\end{center}

\section{Homotopy types of $\match(P_{3,t})$ for small $t$}

\begin{center}
\begin{tabular}{|r|c||r|c||r|c|} \hline
$t$ & homotopy & $t$ & homotopy & $t$ & homotopy \\
    & type     &     & type     &     & type \\ \hline
2 & $\bigvee\limits_{2} S^0$ &
6 & $\bigvee\limits_{4} S^2$  &
10 & $\bigvee\limits_{28} S^3$ \\ \hline
3 & $S^1$ &
7 & $\bigvee\limits_{12} S^2$ &
11 & $\bigvee\limits_{16} S^3 \vee \bigvee\limits_{6} S^4$ \\ \hline
4 & $\bigvee\limits_5 S^1$ &
8 & $\bigvee\limits_8 S^2 \vee S^3$ &
12 & $\bigvee\limits_{38} S^4$ \\ \hline
5 & $\bigvee\limits_4 S^1$ &
9 & $\bigvee\limits_{13} S^3$  &
13 & $\bigvee\limits_{64} S^4 \vee S^5$ \\ \hline
\end{tabular}

\end{center}

\end{document}